\documentclass[preprint]{elsarticle}
\usepackage{amsfonts}
\usepackage{amsmath}
\usepackage[all]{xy}
\usepackage{tikz}
\usetikzlibrary{trees}
\usetikzlibrary{arrows,shapes,snakes,automata,backgrounds,petri}
\newtheorem{thm}{Theorem}
\newtheorem{lemma}{Lemma}[section]
\newtheorem{prop}[lemma]{Proposition}

\newtheorem{rem}[lemma]{Remark}
\newtheorem{example}[lemma]{Example}
\newtheorem{definition}[lemma]{Definition}
\newtheorem{notation}[lemma]{Notation}
\newproof{proof}{Proof}
\newproof{pot1}{Proof of Theorem~\ref{thmtwoprodP0}}
\newproof{pot2}{Proof of Theorem~\ref{thm_main_forest}}
\newproof{pot3}{Proof of Theorem~\ref{thm_main_caterpillar}}

\newcommand{\bunderline}[1]{\mkern2mu\underline{\mkern-2mu#1\mkern-4mu}\mkern4mu }
\def\myd#1{\bunderline{#1}}

\begin{document}
\begin{frontmatter}

\title{Some results on the structure and spectra of matrix-products}

\author[ref1]{Murad Banaji\corref{cor1}}
\author[ref2]{Carrie Rutherford}

\address[ref1]{Department of Mathematics, University of Portsmouth, Lion Gate Building, Lion Terrace, Portsmouth, Hampshire PO1 3HF, UK.}
\address[ref2]{Department of Business Studies, London South Bank University, 103 Borough Road, London SE1 0AA, UK}

\cortext[cor1]{Corresponding author: murad.banaji@port.ac.uk. }

\begin{abstract}
We consider certain matrix-products where successive matrices in the product belong alternately to a particular qualitative class or its transpose. The main theorems relate structural and spectral properties of these matrix-products to the structure of underlying bipartite graphs. One consequence is a characterisation of caterpillars: a graph is a caterpillar if and only if all matrix-products associated with it have real nonnegative spectrum. Several other equivalences of this kind are proved. The work is inspired by certain questions in dynamical systems where such products arise naturally as Jacobian matrices, and the results have implications for the existence and stability of equilibria in these systems. 
\end{abstract}


\end{frontmatter}

\section{Introduction and statement of the main results}

The question of how the structure of a matrix in a combinatorial sense relates to its linear algebraic properties has been intensively studied, particularly in the context of sign nonsingularity (\cite{Thomassen1,LundyMaybeeBuskirk,shao,RobertsonSeymourThomas} to name just a few examples), but also of other questions in linear algebra, both spectral and nonspectral (\cite{maybee,ZeilbergerDiscMath,barkerTam1992} for example). Here we explore how the combinatorial structure of a real matrix $A$, not necessarily square, as encoded in its bipartite graph, relates to properties of matrix-products where successive matrices in the product belong alternately to the qualitative class of $A$ or its transpose.

The results are inspired partly by the study of chemical reaction networks, namely dynamical systems describing the evolution of chemical species undergoing a set of reactions. In this setting the matrices studied are Jacobian matrices, and the bipartite graph from which one wishes to draw conclusions is a natural representation of the chemical system, often termed the ``species-reaction graph'' or ``SR-graph'' \cite{craciun1,banajicraciun}. Under weak assumptions, systems of chemical reactions have Jacobian matrices which factorise as $-AB^{\mathrm{t}}$ where $A$ and $B$ are matrices such that $B$ lies in the closure of the qualitative class of~$A$. The study of these systems thus naturally raises general questions about what can be said about matrix-products where alternate factors belong either to some qualitative class or its transpose.

The main results, Theorems~\ref{thmtwoprodP0}~to~\ref{thm_main_caterpillar}, are easy to state after some definitions. 

\begin{definition}[{\bf Sign-pattern, qualitative class}]
Given $A \in \mathbb{R}^{n \times m}$ define $\mathrm{sign}\,A\in \mathbb{R}^{n \times m}$, the sign-pattern of $A$, as the $(0,1,-1)$-matrix whose entries have the same signs as the entries of $A$; the qualitative class of $A$ is the set of matrices with the same sign-pattern as $A$, i.e., $\mathcal{Q}(A) = \{B \in \mathbb{R}^{n \times m}\colon \mathrm{sign}\,B = \mathrm{sign}\,A\}$. Also useful is $\mathcal{Q}_0(A)$, the topological closure of $\mathcal{Q}(A)$, regarded as a subset of $\mathbb{R}^{n \times m}$. 
\end{definition}
\begin{definition}
Given $A \in \mathbb{R}^{n \times m}$, define
\[
\mathcal{Q}^{k}(A) = \{A_1A^{\mathrm{t}}_2A_3\cdots A^{\mathrm{t}}_{k}\colon A_i \in \mathcal{Q}(A), \,\, i = 1, \ldots, k\}\,,
\]
\[
\mathcal{Q}_0^{k}(A) = \{A_1A^{\mathrm{t}}_2A_3\cdots A^{\mathrm{t}}_{k}\colon A_i \in \mathcal{Q}_0(A), \,\, i = 1, \ldots, k\}\,.
\]
Note that $\mathcal{Q}^{k}(A)\subseteq \mathcal{Q}_0^{k}(A) \subseteq \mathrm{cl}(\mathcal{Q}^{k}(A))$ where $\mathrm{cl}(\mathcal{Q}^{k}(A))$ is the closure of $\mathcal{Q}^{k}(A)$.
\end{definition}
\begin{definition}[{\bf $P_0$-matrices}]
A real square matrix is a $P_0$-matrix (resp. $P$-matrix) if all of its principal minors are nonnegative (resp. positive). We write $\mathbf{P_0}$ for the set of $P_0$-matrices. 
\end{definition}
\begin{definition}[{\bf Matrices with nonnegative real eigenvalues}]
We write $\mathbf{PS}$ for the set of real matrices with real nonnegative spectrum. (Real) positive semidefinite matrices are the symmetric elements of $\mathbf{PS}$.
\end{definition}
\begin{definition}[{\bf Forest, caterpillar forest}] A forest is an acyclic graph; a {\bf caterpillar} is a tree which becomes a path on removal of its leaves; a caterpillar forest is a forest each of whose connected components is a caterpillar. 
\end{definition}

\begin{definition}[{\bf Bipartite graph/weighted bipartite graph of a matrix}]
\label{defbipartite}
Given $A \in \mathbb{R}^{n \times m}$ define $\Gamma_A$, the bipartite graph of $A$, as follows: $A$ is a graph on $n + m$ vertices with bipartition $V(\Gamma_A) = \{X_1, \ldots, X_n\} \cup \{Y_1, \ldots, Y_m\}$, and with edge $X_iY_j$ if and only if $A_{ij} \neq 0$. The weighted bipartite graph of $A$ is $\myd{\Gamma}_A = (\Gamma_A, w)$, where $w:E(\Gamma_A) \to \{0,1\}$ is defined via 
\[
w(X_iY_j) = \left\{\begin{array}{ll}1 & \mbox{ if }\, A_{ij} < 0\,,\\0 & \mbox{ if }\, A_{ij} > 0\,.\end{array}\right.
\]
\end{definition}
\begin{rem}
Weighted graphs will always be depicted underlined to make it clear which results are independent of edge weightings. 
\end{rem}

A weighted graph $\myd{\Gamma}_A$ is ``$2$-odd'' if the difference between the total number of $1$-weighted and $0$-weighted edges in each of its cycles equals $2$ modulo $4$. The main theorems are:
\begin{thm}
\label{thmtwoprodP0}
$\myd{\Gamma}_A$ is 2-odd $\Leftrightarrow$ $\mathcal{Q}^{2}(A) \subseteq \mathbf{P_0}$.
\end{thm}
\begin{thm}
\label{thm_main_forest}
$\Gamma_A$ is a forest $\Leftrightarrow$ $\mathcal{Q}^4(A) \subseteq \mathbf{P_0}$ $\Leftrightarrow$ $\mathcal{Q}^{2}(A) \subseteq \mathbf{PS}$. 
\end{thm}
\begin{thm}
\label{thm_main_caterpillar}
$\Gamma_A$ is a caterpillar forest $\Leftrightarrow$ $\mathcal{Q}^6(A) \subseteq \mathbf{P_0}$ $\Leftrightarrow$ $\mathcal{Q}^4(A) \subseteq \mathbf{PS}$ $\Leftrightarrow$ $\mathcal{Q}^{2k}(A) \subseteq \mathbf{P_0}$ for all $k \in \mathbb{N}$ $\Leftrightarrow$ $\mathcal{Q}^{2k}(A) \subseteq \mathbf{PS}$ for all $k \in \mathbb{N}$.
\end{thm}
Theorem~\ref{thmtwoprodP0} is known but a brief proof comes naturally as a corollary of certain preliminary results needed for the proof of Theorems~\ref{thm_main_forest}~and~\ref{thm_main_caterpillar}. 

\begin{rem}
\label{remP0PSclosed}
It is clear that the sets $\mathbf{P_0}$ and $\mathbf{PS}$ are closed, namely a convergent sequence of matrices of some fixed dimension in $\mathbf{P_0}$ (resp. $\mathbf{PS}$) converges to a matrix in $\mathbf{P_0}$ (resp. $\mathbf{PS}$). Thus $\mathcal{Q}^{2k}(A) \subseteq \mathbf{P_0}$ if and only if $\mathcal{Q}_0^{2k}(A) \subseteq \mathbf{P_0}$, and $\mathcal{Q}^{2k}(A) \subseteq \mathbf{PS}$ if and only if $\mathcal{Q}_0^{2k}(A) \subseteq \mathbf{PS}$. 
\end{rem}

\section{Some definitions and basic observations}

\subsection{Matrices and matrix-sets}

\begin{notation}[{\bf Submatrices and minors}] Given $A \in \mathbb{R}^{n \times m}$ and nonempty sets $\alpha \subseteq \{1, \ldots, n\}$ and $\beta \subseteq \{1,\ldots, m\}$, $A(\alpha|\beta)$ is the submatrix of $A$ with rows from $\alpha$ and columns from $\beta$. If $|\alpha| = |\beta|$, then $A[\alpha|\beta] \stackrel{\text{\tiny def}}{=} \mathrm{det}(A(\alpha|\beta))$, and $A[\alpha]$ is shorthand for the {\bf principal minor} $A[\alpha|\alpha]$. 
\end{notation}

\begin{lemma}[{\bf The Cauchy-Binet formula}] 
Given $A \in \mathbb{R}^{n \times m}$ and $B \in \mathbb{R}^{m \times n}$, and any nonempty $\alpha \subseteq \{1, \ldots, n\}$, $\beta \subseteq \{1, \ldots, m\}$ with $|\alpha| = |\beta|$:
\begin{equation}
\label{eqCB0}
(AB)[\alpha|\beta] = \sum_{\substack{\gamma \subseteq \{1, \ldots, m\}\\ |\gamma| = |\alpha|}}A[\alpha|\gamma]B[\gamma|\beta]. 
\end{equation}
\end{lemma}
\begin{proof}
See \cite{gantmacher}, for example. 
\qquad \qed \end{proof}

\begin{notation}[{\bf Products/functions of matrix-sets}] 
If $\mathcal{A}, \mathcal{B}$ are sets of matrices of suitable dimension then
\[
\mathcal{A}\mathcal{B} = \{AB\colon A \in \mathcal{A}, B \in \mathcal{B}\}\,.
\]
$\mathrm{spec}\,A$ will denote the spectrum of a square matrix $A$ regarded (depending on context) either as a set in $\mathbb{C}$ or a multiset in $\mathbb{C}$. Given a set of matrices $\mathcal{A}$, $\mathrm{spec}\,\mathcal{A}$ is an abbreviation for $\cup_{A \in \mathcal{A}}\mathrm{spec}\,A$, regarded as a set.
\end{notation}

\subsection{Graphs and digraphs}

\begin{definition}[{\bf Walk, subwalk, cycle}]
Following \cite{BondyMurty}, a walk $W$ in a graph $G$ is defined  as a nonempty alternating sequence of vertices and edges, beginning and ending with a vertex, and where each edge in $W$ is preceded and followed by its two end-points. In the case of a digraph each edge is preceded by its start-point and followed by its end-point. The {\bf length} $|W|$ of $W$ is the number of edges in $W$, counted with repetition. If the first and last vertex are the same, the walk is {\bf closed}. We consider two closed walks as equivalent if they differ only in the choice of initial/terminal vertex; by an abuse of notation each equivalence class will be termed a closed walk. In what follows we may refer to walks by their sequence of edges, or their sequence of vertices. A subwalk of $W$ is a walk which is also a subsequence of consecutive entries in $W$ (these are termed ``sections'' in \cite{BondyMurty}). A cycle is a closed walk without repeated vertices (except, naturally, the initial/final vertex). 
\end{definition}

\begin{notation}[{\bf Indices on vertices of a closed walk}] Given a closed walk $(v_0, v_1, \ldots, v_r=v_0)$ of length $r$, all vertex indices are assumed without comment to be reduced modulo $r$. 
\end{notation}

\begin{rem}
\label{remwalkdegree}
Underlying a walk $W$ on a graph $G$ is a connected subgraph $W'$ of $G$. Given a closed walk $W$ such that $W'$ is a tree, if some vertex $v \in V(W')$ has degree $m$ in $W'$, then $v$ occurs {\em at least} $m$ times in $W$. 
\end{rem}

\begin{definition}[{\bf Tree walk, caterpillar walk}]
If the graph underlying a walk $W$ is a tree, we say $W$ is a tree walk. If the graph underlying a walk $W$ is a caterpillar, we say $W$ is a caterpillar walk.
\end{definition}

\begin{definition}[{\bf Weighted digraph of a matrix-product}] 
\label{bcgraph}
Given a square matrix-product $A_1\cdots A_k$ define
\[
M(A_1,\cdots, A_k) = \left(\begin{array}{ccccc}\mathbf{0} &\mathrm{sign}\,A_1 & \mathbf{0} &\cdots & \mathbf{0}\\\mathbf{0} &\mathbf{0} &\mathrm{sign}\,A_2 & \cdots &\mathbf{0} \\ 
\vdots & \vdots & \vdots & \ddots & \vdots\\
\mathbf{0} &\mathbf{0} &\mathbf{0}& \cdots &\mathrm{sign}\,A_{k-1} \\ 
\mathrm{sign}\,A_{k} &\mathbf{0} &\mathbf{0} & \cdots &\mathbf{0}\end{array}\right)
\]
and regard this matrix as the adjacency matrix of a weighted digraph $\myd{G}_{A_1\cdots A_k}$ where an edge has weight $1$ if it corresponds to a negative entry in $M(A_1,\cdots, A_k)$ and weight $0$ if it corresponds to a positive entry in $M(A_1,\cdots, A_k)$. 
\end{definition}

\begin{rem}
$\myd{G}_{A_1\cdots A_k}$ is just the ``signed $(k,\{1\})$-block circulant digraph'' of the matrix-product $A_1\cdots A_k$ as defined in \cite{banajirutherford1} with edge-weights replacing signs to make the computations here more natural. Each edge $e$ of $\myd{G}_{A_1\cdots A_k}$ corresponds to a unique nonzero entry in some $A_j$.
\end{rem}

\begin{notation}
Following the convention noted earlier, we write $G_{A_1\cdots A_k}$ for the unweighted version of $\myd{G}_{A_1\cdots A_k}$.
\end{notation}

\begin{notation}[{\bf Isomorphism}]
We write $G \cong H$ if two (unweighted) graphs $G,H$ are isomorphic, namely have permutation-similar adjacency matrices. A weighted graph $\myd{G}$ defines a signed adjacency matrix $A(\myd{G})$ with negative entries in $A(\myd{G})$ corresponding to edges with weight $1$ in $\myd{G}$. Given two such graphs, $\myd{G} \cong \myd{\Gamma}$ will mean that $\myd{G}$ and $\myd{\Gamma}$ have permutation-similar signed adjacency matrices.
\end{notation}
There is a special case of the weighted digraph of a matrix-product most relevant here: given $A \in \mathbb{R}^{n \times m}$ and $s \in \mathbb{N}$ we will be interested in the digraph $\myd{G}_{(AA^{\mathrm{t}})^s}$. As vertices of both $\myd{G}_{(AA^{\mathrm{t}})^s}$ and $\myd{\Gamma}_A$ correspond to rows/columns of the matrix $A$, and each edge of $\myd{G}_{(AA^{\mathrm{t}})^s}$ corresponds to an entry in $A$, there is a natural association between $\myd{G}_{(AA^{\mathrm{t}})^s}$ and $\myd{\Gamma}_A$.

\begin{definition}[{\bf Projection from digraph to bipartite graph}]
Given a matrix $A \in \mathbb{R}^{n \times m}$, a positive integer $s$, the weighted digraph $\myd{G} = \myd{G}_{(AA^{\mathrm{t}})^s}$, and the weighted bipartite graph $\myd{\Gamma}_A$, define in a natural way the projection $\pi\colon \myd{G} \to \myd{\Gamma}_A$ which takes vertices to vertices and weighted edges to weighted edges. Thus with the notation of Definition~\ref{defbipartite}, if $e \in E(\myd{G})$ corresponds to entry $A_{ij}$ then $\pi(e) = X_iY_j$. 
\end{definition}

\begin{definition}[{\bf Weight of an edge-list}]
Given a weighted (di)graph $\myd{G} = (G, w\colon E(G) \to \{0,1\})$, and any edge-list $E' = (e_1, \ldots, e_k)$ where $e_i \in E(G)$ for each $i$ define
\[
w(E') = \sum_{i = 1}^k w(e_i) \pmod 2
\]
as the weight of $E'$. 
\end{definition}

\begin{rem}[{\bf Weight of a closed tree walk}]
\label{remclosedweight}
If $G$ is any weighted graph and $W$ is a closed tree walk on $G$, then $w(W) = 0$. This is clear, since each edge must be traversed an even number of times. 
\end{rem}

\begin{definition}[{\bf k-weight, k-odd, k-even}]
Given a weighted graph or digraph $\myd{G} = (G, w\colon E(G) \to \{0,1\})$ and some list of edges $E$, define
\[
w_k(E) = |E|/k + w(E) \pmod 2\,,
\]
as the $k$-weight of $E$. Thus, for example, $w_2(E) \in \{0, \frac{1}{2}, 1, \frac{3}{2}\}$. A cycle $C$ is termed $k$-odd if its $k$-weight is $1$ and $k$-even if its $k$-weight is $0$. A weighted graph or digraph $\myd{G}$ is termed $k$-odd if all its cycles are $k$-odd. 
\end{definition}

\begin{rem}
Clearly a necessary condition for a weighted (di)graph to be $2$-odd is for it to be bipartite: otherwise it includes a cycle with non-integer $2$-weight. More generally, each cycle in a graph with $k$-block circulant structure (such as $\myd{G}_{A_1\cdots A_k}$ in Definition~\ref{bcgraph}) has length a multiple of $k$ and hence is either $k$-even or $k$-odd.
\end{rem}

\begin{rem}
\label{remcompound}
Given $A \in \mathbb{R}^{n \times m}$, it is shown in \cite{banajicraciun} that $\myd{\Gamma}_A$ is $2$-odd if and only if all minors of $A$ are signed. This can be phrased elegantly using compound matrices (\cite{muldowney} for example): $\myd{\Gamma}_A$ is $2$-odd if and only if
\[
\Lambda^k\mathcal{Q}(A) \subseteq \mathcal{Q}(\Lambda^k A)
\]
for all $k = 1, \ldots, \min\{n,m\}$. Here $\Lambda^k A$ is the $k$th exterior power (or $k$th multiplicative compound) of $A$, namely the ${n\choose k} \times {m\choose k}$ matrix of $k \times k$ minors of $A$; $\Lambda^k\mathcal{Q}(A)$ is an abbreviation for $\{\Lambda^k B\colon B \in \mathcal{Q}(A)\}$. 
\end{rem}

\begin{definition}
We refer to the following tree as $T^*$: 
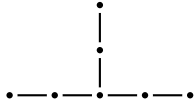
\begin{center}
\begin{tikzpicture}[domain=0:4,scale=0.6]

\path (0,0) coordinate(R1);
\path (1,0) coordinate(R2);
\path (2,0) coordinate(R3);
\path (3,0) coordinate(R4);
\path (4,0) coordinate(R5);
\path (2,1) coordinate(R6);
\path (2,2) coordinate(R7);

\draw[thick] (R1)--(R2);
\draw[thick] (R2)--(R3);
\draw[thick] (R3)--(R4);
\draw[thick] (R4)--(R5);
\draw[thick] (R3)--(R6);
\draw[thick] (R6)--(R7);

\fill[color=white] (R1) circle (5pt);
\fill[color=white] (R2) circle (5pt);
\fill[color=white] (R3) circle (5pt);
\fill[color=white] (R4) circle (5pt);
\fill[color=white] (R5) circle (5pt);
\fill[color=white] (R6) circle (5pt);
\fill[color=white] (R7) circle (5pt);

\fill[color=black] (R1) circle (2pt);
\fill[color=black] (R2) circle (2pt);
\fill[color=black] (R3) circle (2pt);
\fill[color=black] (R4) circle (2pt);
\fill[color=black] (R5) circle (2pt);
\fill[color=black] (R6) circle (2pt);
\fill[color=black] (R7) circle (2pt);

\end{tikzpicture}
\end{center}
As is well-known, a forest is a caterpillar forest if and only if it includes no subgraph isomorphic to $T^*$.
\end{definition}

\section{Preliminary results}

\subsection{Basic properties of $\mathcal{Q}^k(A)$}

\begin{lemma}
\label{lempathcon}
Given $A \in \mathbb{R}^{n \times m}$, for each $k$, $\mathcal{Q}^k(A)$ is path-connected. 
\end{lemma}
\begin{proof}
Given any $B_1,\ldots,B_k$ and $C_1,\ldots,C_k$ such that $B_i, C_i \in \mathcal{Q}(A)$ for odd $i$ and $B_i, C_i \in \mathcal{Q}(A^{\mathrm{t}})$ for even $i$, define $B = \prod_{i=1}^kB_i$, $C = \prod_{i=1}^kC_i$. Then $B, C \in \mathcal{Q}^k(A)$ and $\gamma(t) = \prod_{i=1}^k(tC_i + (1-t)B_i)$, $0 \leq t \leq 1$, defines a path from $B$ to $C$ in $\mathcal{Q}^k(A)$. 
\qquad \qed \end{proof}

\begin{lemma}
\label{lemoddrank}
Let $A \in \mathbb{R}^{n \times m}$ have rank $s$, and let $\myd{\Gamma}_A$ be $2$-odd. For each positive integer $k$: 
\begin{enumerate}
\item All matrices in $\mathcal{Q}^k(A)$ and $\mathcal{Q}^k(A^{\mathrm{t}})$ have rank $s$. 
\item All matrices in $\mathcal{Q}^{2k}(A)$ (resp. $\mathcal{Q}^{2k}(A^{\mathrm{t}})$) have exactly $n-s$ (resp. $m-s$) zero eigenvalues. 
\end{enumerate}
\end{lemma}
\begin{proof}
First, observe that statement (2) follows easily from (1): if some $M \in \mathcal{Q}^{2k}(A)$ has rank $s$ but more than $n-s$ zero eigenvalues, then some power of $M$ must have rank less than $s$, contradicting (1). A similar argument applies to $\mathcal{Q}^{2k}(A^{\mathrm{t}})$. So we need prove only (1). Observe also that for fixed $k$ all matrices in $\mathcal{Q}^k(A)$ have rank $s$ if and only if all matrices in $\mathcal{Q}^k(A^{\mathrm{t}})$ have rank $s$ since each element in $\mathcal{Q}^k(A^{\mathrm{t}})$ is the transpose of an element of $\mathcal{Q}^k(A)$. In what follows $A_i \in \mathcal{Q}(A)$ for odd $i$ and $A_i \in \mathcal{Q}(A^{\mathrm{t}})$ for even $i$.

(i) $k=1$: since $\myd{\Gamma}_A$ is $2$-odd, given $A_1 \in \mathcal{Q}(A)$ and nonempty $\alpha \subseteq \{1, \ldots, n\}$, $\beta \subseteq \{1, \ldots, m\}$ with $|\alpha|=|\beta|$, $\mathrm{sign}(A_1[\alpha|\beta]) = \mathrm{sign}(A[\alpha|\beta])$ (Theorem~11 in \cite{banajicraciun}). In terms of compound matrices, $\Lambda^k(A_1) \in \mathcal{Q}(\Lambda^k(A))$ for each $k$ (see Remark~\ref{remcompound}). That $\mathrm{rank}(A_1) = s$ is now immediate since the rank of a matrix is the order of its largest nonsingular square submatrix. 

(ii) $k=2$: consider any $A_1, A_2^{\mathrm{t}} \in \mathcal{Q}(A)$. By (i) $A_1[\alpha'|\beta']A_2[\beta'|\alpha'] \geq 0$ for all nonempty $\alpha' \subseteq \{1, \ldots, n\}, \beta' \subseteq \{1, \ldots, m\}$ of equal size, and moreover there exist $\alpha \subseteq \{1, \ldots, n\}, \beta \subseteq \{1, \ldots, m\}$ with $|\alpha| = |\beta| = s$ and such that $A_1[\alpha|\beta]A_2[\beta|\alpha] > 0$. Applying the Cauchy-Binet formula, $(A_1A_2)[\alpha] > 0$, i.e., $A_1A_2$ has a nonzero principal minor of order $s$ and so $\mathrm{rank}(A_1A_2) \geq s$. Since $\mathrm{rank}(A_1A_2) \leq s$, we conclude that $\mathrm{rank}(A_1A_2) = s$. Since $A_2$ has rank $s$, this implies that $\mathrm{im}(A_2) \cap \mathrm{ker}(A_1) = \{0\}$. Similarly $(A_2A_1)[\beta] > 0$, so $A_2A_1$ has rank $s$ and $\mathrm{im}(A_1) \cap \mathrm{ker}(A_2) = \{0\}$. 

(iii) To show the result for arbitrary $k$ we proceed by induction. Suppose the result is true for $k = r \geq 2$ and consider a product of the form $A_1A_2\cdots A_{r+1}$. By the inductive hypothesis, $\mathrm{rank}(A_2\cdots A_{r+1}) = s$. But $\mathrm{im}(A_{2}\cdots A_{r+1}) \subseteq \mathrm{im}(A_2)$ and, by (ii), $\mathrm{im}(A_2) \cap \mathrm{ker}(A_1) = \{0\}$. So $\mathrm{rank}(A_1A_2\cdots A_{r+1}) = s$. \qquad \qed \end{proof}

\subsection{$P_0$ matrix-products and $k$-odd digraphs}

We first need to state and develop some results from \cite{banajirutherford1}.

\begin{thm}
\label{main}
Consider any matrices $A_0,A_1,\ldots, A_{k-1}$ of dimensions such that $A_0A_1\cdots A_{k-1} \in \mathbb{R}^{n \times n}$. Then 
\begin{enumerate}
\item $\mathcal{Q}(A_0)\mathcal{Q}(A_1)\cdots \mathcal{Q}(A_{k-1}) \subseteq \mathbf{P_0}$ $\Leftrightarrow$ $\myd{G}_{A_0\cdots A_{k-1}}$ is $k$-odd.
\item $\mathcal{Q}(A_0)\mathcal{Q}(A_1)\cdots \mathcal{Q}(A_{k-1}) \subseteq \mathbf{PS}$ $\Rightarrow$ $\myd{G}_{A_0\cdots A_{k-1}}$ is $k$-odd. 
\end{enumerate}
\end{thm}
\begin{proof}
{\bf (1)} This is a combination of Theorems~1~and~2 in \cite{banajirutherford1} rephrased in the terminology of this paper. 

{\bf (2)} Suppose $\myd{G}_{A_0\cdots A_{k-1}}$ fails to be $k$-odd. Let $C$ be a $k$-even cycle in $\myd{G}_{A_0\cdots A_{k-1}}$ of length $jk$, and suppose the edges of $C$ are $(e_0, e_1, \ldots, e_{jk-1})$. Assume (without loss of generality) that edge $e_m$ corresponds to an entry in $A_{m \bmod k}$ for each $m$, so corresponding to $C$ is a list of matrix entries, say
\[
(A_0)_{t_0t_1}, (A_1)_{t_1t_2}, \ldots, (A_0)_{t_kt_{k+1}}, (A_1)_{t_{k+1}t_{k+2}}, \ldots, (A_{k-1})_{t_{jk-1}t_0}\,.
\]
Note that the indices satisfy the restriction $t_i \neq t_{i + rk \bmod jk}$ for any $r \in \{1, \ldots, j-1\}$ (otherwise $C$ has a repeated vertex and  fails to be a cycle). Since $C$ is $k$-even, $j + \sum_{m=0}^{jk-1} w(e_m) \equiv 0 \bmod 2$, i.e., 
\[
(-1)^j(A_0)_{t_0t_1}(A_1)_{t_1t_2}\cdots (A_{k-1})_{t_{jk-1}t_0} > 0\,.
\]
For $i = 0, \ldots, k-1$, define $\tilde A_i\in \mathcal{Q}_0(A_i)$ via: 
\[
(\tilde A_i)_{pq} = \left\{\begin{array}{ll}\mathrm{sign}((A_i)_{pq}) & \mbox{if  } p=t_r,q=t_{r+1}\,\,\,\mbox{for some}\,\,\, r \equiv i \pmod k\\0&\mbox{otherwise}\end{array}\right. 
\]
(Namely, we replace those entries of $A_i$ which contribute to the cycle $C$ with their signs, and replace the remaining entries with zeros.) Define $B = \tilde {A}_0\cdots \tilde{A}_{k-1}$, and compute:
\[
B_{ij} = \sum_{q_1,q_2,\ldots,q_{k-1}}(\tilde {A}_0)_{iq_1}(\tilde {A}_1)_{q_1q_2}\cdots(\tilde {A}_{k-1})_{q_{k-1}j}\,.
\]
(Here each index $q_i$ ranges over all values such that the expression makes sense.) By the definitions of $\tilde A_i$, there are exactly $j$ nonzero products of this form (each of which necessarily has value $\pm 1$), namely:
\[
(\tilde {A}_0)_{t_{sk}t_{sk+1}}\cdots(\tilde {A}_{k-1})_{t_{(s+1)k-1}t_{(s+1)k}}, \quad s= 0, \ldots, j-1\,.
\]
(Here we define $t_{jk} = t_0$.) In other words:
\[
B_{pq} = \left\{\begin{array}{ll}\prod_{i = 0}^{k-1}(\tilde {A}_i)_{t_{sk+i}t_{sk+i+1}} & \mbox{if  } p=t_{sk},q=t_{(s+1)k}\,\, (s = 0, \ldots, j-1)\\0&\mbox{otherwise}\end{array}\right. 
\]
Note that $t_{rk} \neq t_{sk}$ for $r, s \in \{0, \ldots, j-1\}$, $r \neq s$. In other words, the only nonzero entries in $B$ are $B_{t_0t_k}, B_{t_kt_{2k}}, \ldots, B_{t_{(j-1)k}t_0}$. Immediately this implies that given $\alpha \subseteq \{1, \ldots, n\}$, the principal minor $B[\alpha] = 0$ unless $\alpha = \alpha' \stackrel{\text{\tiny def}}{=} \{t_0, t_k, \ldots, t_{(j-1)k}\}$, in which case 
\begin{eqnarray*}
B[\alpha'] = (-1)^{j+1}\prod_{s=0}^{j-1}B_{t_{sk}t_{(s+1)k}} &=& (-1)^{j+1}\prod_{s=0}^{j-1}\prod_{i = 0}^{k-1}(\tilde {A}_i)_{t_{sk+i}t_{sk+i+1}}\\ &=& (-1)^{j+1}(\tilde A_0)_{t_0t_1}(\tilde A_1)_{t_1t_2}\cdots (\tilde A_{k-1})_{t_{jk-1}t_0}\,.
\end{eqnarray*}
Since $(-1)^j(A_0)_{t_0t_1}(A_1)_{t_1t_2}\cdots (A_{k-1})_{t_{jk-1}t_0} > 0$, $(-1)^jB[\alpha'] = (-1)^{j+1}$. Thus the characteristic polynomial of $B$ is $\lambda^{n-j}(\lambda^j + (-1)^jB[\alpha']) = \lambda^{n-j}(\lambda^j + (-1)^{j+1})$. For $j=1$, the unique nonzero eigenvalue of $B$ is $-1$, while for $j \geq 2$, the nonzero eigenvalues of $B$ are just the $j$th roots of $(-1)^{j}$ which clearly do not all lie on the nonnegative real axis (in fact, $-1$ is a root in each case). Thus $\mathcal{Q}_0(A_0)\mathcal{Q}_0(A_1)\cdots\mathcal{Q}_0(A_{k-1}) \not \subseteq \mathbf{PS}$. As $\mathbf{PS}$ is closed, $\mathcal{Q}(A_0)\mathcal{Q}(A_1)\cdots\mathcal{Q}(A_{k-1}) \not \subseteq \mathbf{PS}$.
\qquad \qed \end{proof}

\begin{rem}
The proof makes it clear that the second claim of Theorem~\ref{main} could in fact be strengthened to $\mathcal{Q}(A_0)\mathcal{Q}(A_1)\cdots \mathcal{Q}(A_{k-1}) \subseteq \mathbf{PS}$ implies that $\myd{G}_{A_0\cdots A_{k-1}}$ is  $k$-odd and moreover has no cycles of length $jk$ where $j \geq 2$. This follows as the $j$th roots of $\pm 1$ are clearly not all real and positive for $j \geq 2$. 
\end{rem}

The next two results are some consequences of Theorem~\ref{main} needed here:

\begin{prop}
\label{propnotP0notPS}
Let $A$ be a real matrix. Then:
\[
\mathcal{Q}^{2k}(A) \subseteq \mathbf{PS} \,\,\, \Rightarrow\,\,\,\myd{G}_{(AA^{\mathrm{t}})^k}\,\,\mbox{ is $2k$-odd }\,\, \Leftrightarrow \,\,\, \mathcal{Q}^{2k}(A) \subseteq \mathbf{P_0}\,.
\]
\end{prop}
\begin{proof}
This is the specialisation of Theorem~\ref{main} where consecutive matrices in the product are just $A$ or $A^{\mathrm{t}}$.
\qquad \qed \end{proof}

\begin{prop}
\label{propnotP0PS}
Let $B$ be a submatrix of $A \in \mathbb{R}^{n \times m}$. Then
\begin{enumerate}
\item $\mathcal{Q}^{2k}(B) \not \subseteq \mathbf{P_0}$ $\Rightarrow$ $\mathcal{Q}^{2k}(A) \not \subseteq \mathbf{P_0}$ 
\item $\mathcal{Q}^{2k}(B) \not \subseteq \mathbf{PS}$ $\Rightarrow$ $\mathcal{Q}^{2k}(A) \not \subseteq \mathbf{PS}$
\end{enumerate}
\end{prop}
\begin{proof}
(1) By Proposition~\ref{propnotP0notPS}, if $\mathcal{Q}^{2k}(B) \not \subseteq \mathbf{P_0}$ then $\myd{G}_{(BB^{\mathrm{t}})^k}$ fails to be $k$-odd. But $\myd{G}_{(BB^{\mathrm{t}})^k}$ is a subgraph of $\myd{G}_{(AA^{\mathrm{t}})^k}$, and thus $\myd{G}_{(AA^{\mathrm{t}})^k}$ fails to be $k$-odd which, by Proposition~\ref{propnotP0notPS}, implies that $\mathcal{Q}^{2k}(A) \not \subseteq \mathbf{P_0}$. 

(2) Let $B=A(\alpha|\beta)$ where $\emptyset \neq \alpha \subseteq \{1, \ldots, n\}$ and $\emptyset \neq \beta \subseteq \{1, \ldots, m\}$. Suppose $\mathcal{Q}^{2k}(B) \not \subseteq \mathbf{PS}$ and choose $B_i \in \mathcal{Q}(B)$ for $i = 1, \ldots, 2k$ such that $\tilde B \stackrel{\text{\tiny def}}{=} B_1B_2^{\mathrm{t}}\cdots B_{2k-1}B_{2k}^{\mathrm{t}} \not \in \mathbf{PS}$. Define $A_i \in \mathcal{Q}_0(A)$ via $A_i(\alpha|\beta)=B_i$ and $(A_i)_{jl} = 0$ if $j \not \in \alpha$ or $l \not \in \beta$, and let $\tilde A \stackrel{\text{\tiny def}}{=} A_1A_2^{\mathrm{t}}\cdots A_{2k-1}A_{2k}^{\mathrm{t}}$. Clearly $\tilde A(\alpha|\alpha) = \tilde B$ and $\tilde A_{ij} = 0$ if $i \not \in \alpha$ or $j \not \in \alpha$, and consequently 
\[
\mathrm{det}(\lambda I - \tilde A) = \lambda^{n-|\alpha|}\mathrm{det}(\lambda I - \tilde B)\,.
\]
Thus the nonzero spectrum of $\tilde A$ is just that of $\tilde B$, and consequently $\tilde A \not \in \mathbf{PS}$. Since $\mathbf{PS}$ is closed and $\tilde A \in \mathcal{Q}_0^{2k}(A)$, $\mathcal{Q}^{2k}(A) \not \subseteq \mathbf{PS}$.
\qquad \qed \end{proof}


\subsection{$P_0$ matrix-products and the bipartite graph}

We now develop sufficient conditions for $\myd{G}_{(AA^{\mathrm{t}})^k}$ to be $k$-odd based on examination of $\myd{\Gamma}_A$, a smaller and more natural object to associate with $A$. In order to do this we need to be able to relate cycles in $G_{(AA^{\mathrm{t}})^k}$ with closed walks on $\Gamma_A$. Proposition~\ref{sbadrealcycle} below tells us how to identify those closed walks on $\Gamma_A$ which are the projections of cycles in $G_{(AA^{\mathrm{t}})^k}$. The following construction is convenient:

\begin{definition}[$k$-labelled walk]
Given a graph $\Gamma$, and $j,k \in \mathbb{N}$, let $W = (v_0,\,v_1,\,\cdots, v_{j} = v_0)$ be a closed walk on $\Gamma$. Now for $t = 0, \ldots, j-1$, assign to vertex $v_t$ the label $t \bmod k$, so that at the end of the walk, vertex $v_t$ has a list of labels $l_t$, each belonging to $\{0, \ldots, k-1\}$. If vertex $v_t$ occurs $r$ times in the list $(v_0, \ldots, v_{j-1})$ then $l_t$ is a list of $r$ labels, which may or may not all be distinct. We refer to the list of labelled vertices
\[
W_k \stackrel{\text{\tiny def}}{=} ((v_0, l_0),\,(v_1,l_1),\,\cdots, )
\]
as a $k$-labelled walk. 
\end{definition}

\begin{definition}[{\bf $\mathbf{\mathit k}$-repeating, repeating}]
Given $k \in \mathbb{N}$, a closed walk $W$ on a graph $\Gamma$ is {\bf $\mathbf{\mathit k}$-repeating} if some vertex in the $k$-labelled walk $W_k$ has a repeated label. In other words, $W$ has a closed subwalk $W' \neq W$ of length a positive multiple of $k$. A graph $\Gamma$ is {\bf $\mathbf{\mathit k}$-repeating} if for all $j \geq 2$, every closed walk on $\Gamma$ of length $jk$ is $k$-repeating. $\Gamma$ is {\bf repeating} if it is $k$-repeating for all even $k \in \mathbb{N}$. 
\end{definition}

\begin{rem}
\label{rem1rep}
If $\Gamma$ is bipartite, then any closed walk on $\Gamma$ has even length. Thus all closed walks on a bipartite graph are $2$-repeating except those where no vertex is revisited, namely cycles and trivial walks of the form $(uvu)$. 
\end{rem}

\begin{prop}
\label{sbadrealcycle}
A closed walk $W$ on $\Gamma_A$ of length $2jk$ is $2k$-repeating if and only if there does not exist a cycle $L$ in $G_{(AA^{\mathrm{t}})^k}$ such that $W = \pi(L)$. 
\end{prop}
\begin{proof}
Let $W = (v_0, v_1, \ldots, v_{2jk} = v_0)$ be a closed walk in $\Gamma_A$ and let $L = (u_0, u_1, \ldots, u_{2jk})$ be any walk in $G = G_{(AA^{\mathrm{t}})^k}$ such that $W = \pi(L)$. If $W$ is $2k$-repeating, then there exist $i$ and $0 < r < j$ such that $v_i = v_{i+2rk}$, i.e., $\pi(u_i) = \pi(u_{i+2rk})$. But $u_i$ and $u_{i+2rk}$ both belong to the same member of the partition of $V(G)$, and the projection $\pi$ restricted to any member of the partition is injective, so in fact $u_i = u_{i+2rk}$. Consequently, $L$ is not a cycle. Conversely, as $|W|$ is a multiple of $2k$, it is easy to see that there exists a closed walk $L$ in $G$ such that $W = \pi(L)$. Suppose some such $L$ is not a cycle, so there exist $i$ and $0 < r < j$ such that $u_i = u_{i+2rk}$ (the instances of a repeated vertex belong to the same member of the partition and so appear a multiple of $2k$ apart); trivially $\pi(u_i) = \pi(u_{i+2rk})$ and $W$ is $2k$-repeating. 
\qquad \qed \end{proof}

We now have the following corollary of the claim in Proposition~\ref{propnotP0notPS} that $\mathcal{Q}^{2k}(A) \subseteq \mathbf{P_0}$ if and only if $\myd{G}_{(AA^{\mathrm{t}})^k}$ is $2k$-odd:

\begin{prop}
\label{kgoodor2kodd}
$\mathcal{Q}^{2k}(A) \subseteq \mathbf{P_0}$ if and only if every closed walk of length $2jk$ ($j \in \mathbb{N}$) on $\myd{\Gamma}_A$ is either $2k$-repeating or $2k$-odd. 
\end{prop}
\begin{proof}
The conclusion follows from Proposition~\ref{propnotP0notPS} after we show that $\myd{G}_{(AA^{\mathrm{t}})^k}$ is $2k$-odd if and only if every closed walk of length $2jk$ ($j \in \mathbb{N}$) on $\myd{\Gamma}_A$ is either $2k$-repeating or $2k$-odd. 

By proposition~\ref{sbadrealcycle} every closed walk of length $2jk$ on $\myd{\Gamma}_A$ which fails to be $2k$-repeating is the projection of a cycle of length $2jk$ in $\myd{G}_{(AA^{\mathrm{t}})^k}$. Thus if $\myd{G}_{(AA^{\mathrm{t}})^k}$ is $2k$-odd, then every closed walk of length $2jk$ on $\myd{\Gamma}_A$ which fails to be $2k$-repeating must be $2k$-odd. Conversely, suppose every closed walk of length $2jk$ on $\myd{\Gamma}_A$ which fails to be $2k$-repeating is $2k$-odd. Since every cycle in $\myd{G}_{(AA^{\mathrm{t}})^k}$ projects to such a walk, $\myd{G}_{(AA^{\mathrm{t}})^k}$ is $2k$-odd.
\qquad \qed \end{proof}

The machinery so far allows a rapid proof of Theorem~\ref{thmtwoprodP0}, which is also easily inferred from Corollary~13 of \cite{banajicraciun}, or directly from \cite{banajirutherford1}.

\begin{center}
\fbox{%
\parbox{7cm}{%
{\bf Theorem~\ref{thmtwoprodP0}.} $\mathcal{Q}^{2}(A) \subseteq \mathbf{P_0}$ $\Leftrightarrow$ $\myd{\Gamma}_A$ is 2-odd.
\hfil}}
\end{center}

\begin{pot1}
A closed walk in $\Gamma_A$ fails to be $2$-repeating if and only if it is a cycle or is of the form $(uvu)$ (Remark~\ref{rem1rep}); the latter are trivially $2$-odd. Thus $\myd{\Gamma}_A$ is $2$-odd if and only if every closed walk of even length on $\myd{\Gamma}_A$ is either $2$-repeating or $2$-odd. Applying Proposition~\ref{kgoodor2kodd} with $k=1$, this is equivalent to $\mathcal{Q}^{2}(A) \subseteq \mathbf{P_0}$.
\qquad \qed \end{pot1}

We now state a sufficient condition on the bipartite graph of a matrix $A$ to ensure that $\mathcal{Q}^{2k}(A) \subseteq \mathbf{P_0}$. This criterion is central to the proofs of Theorems~\ref{thm_main_forest}~and~\ref{thm_main_caterpillar}.
\begin{prop}
\label{P_0prods}
Let $A \in \mathbb{R}^{n \times m}$ and fix $k \in \mathbb{N}$. Suppose that (i) Every closed walk on $\myd{\Gamma}_A$ has weight $0$, and (ii) $\myd{\Gamma}_A$ is $2k$-repeating. Then $\mathcal{Q}^{2k}(A) \subseteq \mathbf{P_0}$ and $\mathcal{Q}^{2k}(A^{\mathrm{t}})\subseteq \mathbf{P_0}$.
\end{prop}
\begin{proof}
Let $\myd{G} = \myd{G}_{(AA^{\mathrm{t}})^k}$, and let $\pi$ be the projection from $\myd{G}$ to $\myd{\Gamma}_A$ described earlier. The result follows from Proposition~\ref{propnotP0notPS} if we show that (i) and (ii) imply that $\myd{G}$ is $2k$-odd. All cycles in $G$ have length which is a multiple of $2k$. Let $L$ be a cycle in $G$ of length $2jk$ and $W = \pi(L)$; clearly $W$ is a closed walk of length $2jk$ on $\myd{\Gamma}_A$ and $w_{2k}(W) = w_{2k}(L)$. By (ii), if $j > 1$, then $W$ is $2k$-repeating, and consequently, by Proposition~\ref{sbadrealcycle}, is not the projection of a cycle in $G$ contradicting the assumption. So $j=1$, namely, (ii) implies that there are no cycles in $G$ of length greater than $2k$.

By (i), $w(W) = 0$, so $w_{2k}(L) = w_{2k}(W) = j + w(W) \equiv 1 \pmod 2$. Since $L$ was arbitrary, $\myd{G}$ is $2k$-odd. The same argument works equally for $\mathcal{Q}^{2k}(A^{\mathrm{t}})$ since $\myd{\Gamma}_{A^{\mathrm{t}}} \cong \myd{\Gamma}_A$.
\qquad \qed \end{proof}

\subsection{Repeating properties of graphs}

This subsection contains claims about a graph $G$ including:
\begin{enumerate}
\item $G$ is $2$-repeating if and only if $G$ is a forest;
\item $G$ is $4$-repeating if and only if $G$ is a forest;
\item $G$ is $6$-repeating if and only if $G$ is $2s$-repeating for all $s \in \mathbb{N}$ if and only if $G$ is a caterpillar forest.
\end{enumerate}

\begin{prop}
\label{propinherit}
A graph which fails to be $2s$-repeating for some $s \in \mathbb{N}$, fails to be $2r$-repeating for all $r \geq s$. 
\end{prop}
\begin{proof}
The proof is inductive. Let $G$ be a graph which fails to be $2s$-repeating. Consequently there exist $j \geq 2$ and a closed walk $W = (v_0, v_1, \ldots, v_{2js} = v_0)$ which fails to be $2s$-repeating. In other words the $2s$-labelled walk $W_{2s}$ has no repeated label in any label-list. Now consider the walk $\tilde{W}$ constructed from $W$ by adding in a second copy of each pair of the form $v_{2rs-2},v_{2rs-1}$ ($r = 1, \ldots, j$) immediately after it occurs in $W$. Clearly $\tilde{W}$ is a closed walk of length $2j(s+1)$ in $G$. It is straightforward to see that $\tilde{W}$ fails to be $2(s+1)$-repeating. Consider the $2(s+1)$-labelled walk $\tilde{W}_{2(s+1)}$. For each $r$, the vertices $v_{2rs-2},v_{2rs-1}$ acquire new labels $2s$ and $2s+1$ respectively; all other label-sets remain the same. The only possible repeated labels in $\tilde{W}_{2(s+1)}$ are the new labels $2s$ or $2s+1$; but a label $2s$ or $2s+1$ is repeated in $\tilde{W}_{2(s+1)}$ if and only if $2s-2$ or $2s-1$ was repeated in $W_{2s}$, which did not occur by assumption. 
\qquad \qed \end{proof}

\begin{rem}
\label{reminherit}
Observe that if the graph $G$ in the proof of Proposition~\ref{propinherit} is weighted, then the walk $\tilde{W}$ is $2(s+1)$-odd if and only if $W$ is $2s$-odd since the added paths each traverse a single edge twice and thus have weight zero.
\end{rem}

\begin{prop}
\label{propcyclenotrep}
A cycle fails to be $2s$-repeating for all $s \in \mathbb{N}$.
\end{prop}
\begin{proof}
Let $C = (v_0, v_1, \ldots, v_k = v_0)$ be a cycle. $C$ fails to be $2$-repeating: if $C$ is of even length it has length at least $4$ and is itself a closed walk of even length which trivially fails to be $2$-repeating (since no vertex is visited more than once). Otherwise if $C$ has odd length, the walk $W = (v_0, v_1, \ldots, v_0, v_1, \ldots, v_0)$ which traverses $C$ twice is a closed walk of even length at least $4$ which fails to be $2$-repeating. The claim now follows from Proposition~\ref{propinherit}.
\qquad \qed \end{proof}

\begin{prop}
\label{forest4rep}
For a graph $G$, the following are equivalent:
\begin{enumerate}
\item $G$ is a forest.
\item $G$ is $2$-repeating.
\item $G$ is $4$-repeating.
\end{enumerate}
\end{prop}
\begin{proof}
(1) $\Rightarrow$ (3): if $G$ is a forest, then it is $4$-repeating. To see this, let $G$ be a forest, fix $j \geq 2$ and let $W$ be any $4$-labelled closed walk on $G$ of length $4j$. Since each vertex gets either odd or even labels, if any vertex occurs more than two times in $W$ then it must get a repeated label and we are done. So suppose that no vertex occurs more than twice in $W$. Then (Remark~\ref{remwalkdegree}) no vertex has degree more than two in $W'$ the subgraph of $G$ underlying $W$. Since $W'$ is a connected subgraph of $G$, it is a tree, and since the maximum degree of any vertex in $W'$ is $2$, it is a path. Moreover the degree of a vertex in $W'$ must be precisely the number of times it has been visited in $W$; so $W'$ is necessarily a path of length $2j$, say $v_0\cdots v_{2j}$. The only walk from $v_0$ to $v_{2j}$ and back to $v_0$ which doesn't visit any vertex three times is $(v_0, v_1, \ldots, v_{2j-1}, v_{2j}, v_{2j-1}, \ldots, v_0)$ and so (upto a shift of initial vertex), this must be $W$. Clearly vertex $v_{2r}$ receives the same label twice for each $1 \leq r \leq j-1$. Since $j \geq 2$, there is at least one such vertex. (3) $\Rightarrow$ (2): if any graph $G$ is $4$-repeating, then it is $2$-repeating by Proposition~\ref{propinherit}. Finally, (2) $\Rightarrow$ (1): if any graph $G$ is $2$-repeating, then it contains no cycles by Proposition~\ref{propcyclenotrep} with $s = 1$, namely, it is a forest. 
\qquad \qed \end{proof}

\begin{lemma}
\label{notall6rep}
Not every forest is $6$-repeating. In particular, $T^*$ is not $6$-repeating. 
\end{lemma}
\begin{proof}
Consider $T^*$ with vertices labelled as follows:

\begin{center}
\begin{tikzpicture}[domain=0:4,scale=0.6]

\path (0,0) coordinate(R1);
\path (1,0) coordinate(R2);
\path (2,0) coordinate(R3);
\path (3,0) coordinate(R4);
\path (4,0) coordinate(R5);
\path (2,1) coordinate(R6);
\path (2,2) coordinate(R7);

\draw[thick] (R1)--(R2);
\draw[thick] (R2)--(R3);
\draw[thick] (R3)--(R4);
\draw[thick] (R4)--(R5);
\draw[thick] (R3)--(R6);
\draw[thick] (R6)--(R7);

\fill[color=white] (R1) circle (8pt);
\fill[color=white] (R2) circle (8pt);
\fill[color=white] (R3) circle (8pt);
\fill[color=white] (R4) circle (8pt);
\fill[color=white] (R5) circle (8pt);
\fill[color=white] (R6) circle (8pt);
\fill[color=white] (R7) circle (8pt);

\node at (R1) {$v_1$};
\node at (R2) {$v_2$};
\node at (R3) {$v_3$};
\node at (R6) {$v_4$};
\node at (R7) {$v_5$};
\node at (R4) {$v_6$};
\node at (R5) {$v_7$};

\end{tikzpicture}
\end{center}
$W = (v_1v_2v_3v_4v_5v_4v_3v_6v_7v_6v_3v_2v_1)$ is a closed walk of length $12$ on $T^*$ which, by observation, fails to be $6$-repeating. 
\qquad \qed \end{proof}

The following theorem may be of some interest in its own right. 

\begin{thm}[Caterpillar intermediate value theorem]
\label{catIVT}
Let $W$ be a closed caterpillar walk on a graph $G$. Then $W$ includes a closed subwalk of length $2s$ for each $s = 1, \ldots, |W|/2$. 
\end{thm}
\begin{proof}
Let $W'$ be the subgraph of $G$ underlying $W$. The claim makes sense as $W'$ is bipartite and hence $|W|$ is even. Set $r = |W|/2$, and let $W = (v_0, v_1, \cdots v_{2r-1}, v_0)$. All indices on vertices $v_j$ are reduced modulo $2r$.

Let the vertices on some path of maximal length in $W'$ be $u_1, \ldots, u_n$ with $u_i$ adjacent to $u_{i\pm 1}$ for $i = 2, \ldots, n-1$. (Note that $u_1$ and $u_n$ may not be uniquely defined, but the argument is unaffected.) The path $(u_1, \ldots, u_n)$ will be the termed the ``spine'' of $W'$ with its vertices being ``spinal'' while the remaining vertices are ``non-spinal''. Define a vertex labelling $l:V(W') \to \{1, \ldots, n\}$ by $l(v_k) = i$ if either $v_k =u_i$ or $v_k$ is non-spinal but adjacent to $u_i$. By construction 
\begin{equation}
\label{eqdiff}
|l(v_{i+1}) - l(v_i)| \leq 1\,.
\end{equation}
For $s = 1, \ldots, r$ define $f_s\colon \{0, \ldots, 2r-1\} \to \{1, \ldots, n\}$ by $f_s(i) = l(v_{i+2s}) - l(v_i)$. It is straightforward that $f_s(i)$ is odd if and only if exactly one of $v_i$, $v_{i+2s}$ is spinal; otherwise there would exist a closed walk of odd length on $W'$, namely the path $(v_i, \cdots, v_{i+2s})$ followed by the shortest path from $v_{i+2s}$ to $v_i$. Also for each $s$, 
\begin{equation}
\label{sumzero}
\sum_{i=0}^{2r-1} f_s(i) = 0\,.
\end{equation}

Fix $s$ and $i$. Suppose $f_s(i) > 0$. Then either (i) $f_s(i) \geq 2$ in which case $f_s(i+1) \geq 0$ by (\ref{eqdiff}); or (ii) $f_s(i) = 1$ in which case, by the remark above, one of $v_i$, $v_{i+s}$ is spinal and one is non-spinal. In this case $f_s(i+1) \geq 0$ since one of the equalities $l(v_{i+1}) = l(v_i)$ or $l(v_{i+2s+1}) = l(v_{i+2s})$ holds. So in both cases $f_s(i) > 0$ implies $f_s(i+1) \geq 0$. Similarly $f_s(i) < 0$ implies $f_s(i+1) \leq 0$. 

Thus the function $f_s$ cannot take both positive and negative values without also taking the value $0$, and by (\ref{sumzero}), for each $s$, there must exist $j \in \{0, \ldots, 2r-1\}$ such that $f_s(j) = 0$. By the remark above, it is not possible that exactly one of $v_j$, $v_{j+2s}$ is spinal. If both $v_j$ and $v_{j+2s}$ are spinal then they must be the same vertex and $(v_j \cdots v_{j+2s})$ is a closed subwalk of $W$ of length $2s$. If both are non-spinal, then $v_{j-1}$ and $v_{j+2s-1}$ must be spinal and moreover $f_s(j-1) = 0$; $(v_{j-1} \cdots v_{j+2s-1})$ is then a closed subwalk of $W$ of length $2s$. Thus in each case there is a closed subwalk of $W$ of length $2s$. 
\qquad \qed \end{proof}

As a corollary of Theorem~\ref{catIVT} we have the following characterisation of caterpillars:
\begin{prop}
A connected graph $G$ is a caterpillar if and only if each closed walk $W$ on $G$ contains closed subwalks of length $2r$ for every integer $1 \leq r \leq |W|/2$.
\end{prop}
\begin{proof}
Let $G$ be a connected graph. If $G$ is a caterpillar, then by Theorem~\ref{catIVT} any closed walk $W$  contains closed subwalks of length $2r$ for every integer $1 \leq r \leq |W|/2$. If $G$ is not a caterpillar, then either (i) it contains a cycle, and thus a closed walk of length $3$ or more with no closed subwalks (see also Proposition~\ref{propcyclenotrep}), or (ii) it contains the subgraph $T^*$ and thus a closed walk of length $12$ with no closed subwalk of length $6$ (proof of Lemma~\ref{notall6rep}). \qquad \qed \end{proof}

Phrased in terms of repeating properties of graphs we have:
\begin{prop}
\label{onlycatssgood}
For a graph $G$, the following are equivalent:
\begin{enumerate}
\item $G$ is a caterpillar forest.
\item $G$ is repeating.
\item $G$ is $6$-repeating.
\end{enumerate}
\end{prop}
\begin{proof}
(1) $\Rightarrow$ (2). Suppose $G$ is a caterpillar forest, and given $s,j \in \mathbb{N}$ with $j \geq 2$, let $W = (v_1, \ldots, v_{2js})$ be a closed walk on $G$. Since $W$ is a closed caterpillar walk, by Theorem~\ref{catIVT}, $W$ includes a closed subwalk of length $2s < 2js$. Thus $W$ is $2s$-repeating, and since $s,j$ were arbitrary, $G$ is repeating. (2) $\Rightarrow$ (3). If $G$ is repeating, then by definition it is $6$-repeating. (3) $\Rightarrow$ (1). If $G$ is not a caterpillar forest then either (i) it contains a cycle in which case it fails to be $6$-repeating by Proposition~\ref{propcyclenotrep}; or (ii) it contains the subgraph $T^*$ which, by Lemma~\ref{notall6rep}, fails to be $6$-repeating. 
\qquad \qed \end{proof}

\section{Proof of Theorem~\ref{thm_main_forest}}

\begin{definition}
For $A \in \mathbb{R}^{n \times m}$, define
\[
\mathcal{Q}'(A) \stackrel{\text{\tiny def}}{=} \{DAE \colon D,E\,\,\mbox{are positive diagonal matrices}\}\,.
\]
\end{definition}
Clearly $\mathcal{Q}'(A) \subseteq \mathcal{Q}(A)$. In fact:
\begin{thm}
\label{treequalc}
$\mathcal{Q}'(A) = \mathcal{Q}(A)$ if and only if $\Gamma_A$ is a forest.
\end{thm}
\begin{proof}
First, observe that given $A \in \mathbb{R}^{n \times m}$ we can write
\[
\mathcal{Q}(A) = \{P \circ A \colon P \in \mathbb{R}^{n \times m}, P_{ij} > 0\}\,,
\]
where $P \circ A$ means the entrywise product or ``Hadamard product'' of $P$ and $A$. 

{\bf 1) If $\Gamma_A$ is a forest then $\mathcal{Q}'(A) = \mathcal{Q}(A)$.} Since $\mathcal{Q}'(A) \subseteq \mathcal{Q}(A)$ we need only prove that $\mathcal{Q}(A) \subseteq \mathcal{Q}'(A)$. The proof is inductive. 

Note first that the result is trivially true if $A \in \mathbb{R}^{1 \times 1}$, i.e., if $\Gamma_A$ consists of a single pair of vertices, which may or may not be adjacent. Suppose that $A_0 \in\mathbb{R}^{n \times m}$ is such that $\mathcal{Q}'(A_0) = \mathcal{Q}(A_0)$, and consequently, for each positive $P \in \mathbb{R}^{n \times m}$, there exist positive diagonal matrices $D_P$ and $E_P$ such that
\[
P \circ A_0 = D_PA_0E_P\,.
\]
We proceed by augmenting $A_0$ with a new column containing at most one nonzero entry. Assume, without loss of generality, that this is the final column. 

\begin{enumerate}
\item[(a)] Define $A = [A_0|0]$. This amounts to creating $\Gamma_A$ from $\Gamma_{A_0}$ by adding an isolated vertex $Y_{m+1}$. Given any positive $P' \in \mathbb{R}^{n \times (m+1)}$, let $P$ be the matrix consisting of the first $m$ columns of $P'$, namely $P=P'(\{1, \ldots, n\}|\{1, \ldots, m\})$. Defining $D_{P'} = D_P$ and 
\[
E_{P'} = \left(\begin{array}{cc}E_P&0\\0&1\end{array}\right)\,,
\]
we easily compute that $P' \circ A = D_{P'}AE_{P'}$.
\item[(b)] Define $A = [A_0|te_k]$ for arbitrary $t \in \mathbb{R}$ and $k \in \{1, \dots, n\}$. This amounts to creating $\Gamma_A$ from $\Gamma_{A_0}$ by adding a leaf $Y_{m+1}$ adjacent to $X_k$. Given any positive $P' \in \mathbb{R}^{n \times (m+1)}$, let $P$ be defined as before. Then defining $D_{P'} = D_P$ and 
\[
E_{P'} = \left(\begin{array}{cc}E_P&0\\0&P'_{k,m+1}/(D_P)_{kk}\end{array}\right)\,,
\]
we compute that $P' \circ A = D_{P'}AE_{P'}$.
\end{enumerate}
Thus if $A_0$ is augmented with a column containing at most one nonzero entry to give a new matrix $A$, then $\mathcal{Q}'(A) = \mathcal{Q}(A)$. The argument where $A_0$ is augmented with a new row with at most one nonzero entry is similar. 

Clearly any matrix $A$ such that $\Gamma_A$ is a forest can be constructed from a $1 \times 1$ matrix but successively adding rows/columns with at most one nonzero entry (i.e., by successive addition of either leaves or isolated vertices to the bipartite graph). Thus, inductively, if $\Gamma_A$ is a forest then $\mathcal{Q}'(A) = \mathcal{Q}(A)$. 

{\bf 2) If $\mathcal{Q}'(A) = \mathcal{Q}(A)$ then $\Gamma_A$ is a forest.} If $\Gamma_A$ is not a forest then it contains a cycle of length $2r$ with $r \geq 2$, say $(X_{i_0}Y_{j_0}\cdots X_{i_{r-1}}Y_{j_{r-1}}X_{i_0})$. Given positive $P \in \mathbb{R}^{n \times m}$, a necessary condition for the existence of positive diagonal matrices $D$ and $E$ such that $P \circ A = DAE$ is that the $2r$ equations
\[
P_{i_0j_0} = D_{i_0i_0}E_{j_0j_0}, \,\,\,P_{i_1j_0} = D_{i_1i_1}E_{j_0j_0}, \,\,\cdots, P_{i_0j_{r-1}} = D_{i_0i_0}E_{j_{r-1}j_{r-1}}
\]
can all be satisfied. Taking products of alternate equations we get
\[
P_{i_0j_0}P_{i_1j_1}\cdots P_{i_{r-1}j_{r-1}} = P_{i_1j_0}P_{i_2j_1}\cdots P_{i_0j_{r-1}}\,,
\]
and choosing any $P$ not satisfying this equation gives us an element $P\circ A$ of $\mathcal{Q}(A)$ not in $\mathcal{Q}'(A)$. \qquad \qed \end{proof}

We can now prove Theorem~\ref{thm_main_forest} which we restate for readability:
\begin{center}
\fbox{%
\parbox{7cm}{%
{\bf Theorem~\ref{thm_main_forest}.} The following are equivalent:
\begin{itemize}
\item $\Gamma_A$ is a forest.
\item $\mathcal{Q}^4(A) \subseteq \mathbf{P_0}$. 
\item $\mathcal{Q}^{2}(A) \subseteq \mathbf{PS}$. 
\end{itemize}
}}
\end{center}

\begin{pot2}
Let $A \in \mathbb{R}^{n \times m}$. 

{\bf 1.} Suppose $\Gamma_A$ is a forest. Since $\Gamma_A$ is $4$-repeating by Proposition~\ref{forest4rep}, and all closed walks on $\Gamma_A$ have weight $0$ (see Remark~\ref{remclosedweight}), the claim that $\mathcal{Q}^4(A) \subseteq \mathbf{P_0}$ follows from Proposition~\ref{P_0prods} with $s=2$. Now consider arbitrary $A_1,A_2 \in \mathcal{Q}(A)$, so that $A_1A_2^{\mathrm{t}} \in \mathcal{Q}^2(A)$. From Theorem~\ref{treequalc} we can write $A_1A_2^{\mathrm{t}} = D_1AD_2A^{\mathrm{t}}D_3$ where $D_1,D_2,D_3$ are positive diagonal matrices. Defining $M = D_3^{1/2}D_1^{1/2}AD_2^{1/2}$ and $P = D_1^{1/2}D_3^{-1/2}$, it is easy to check that $A_1A_2^{\mathrm{t}} = PMM^{\mathrm{t}}P^{-1}$. In other words $A_1A_2^{\mathrm{t}}$ is (diagonally) similar to a positive semidefinite matrix. Since $A_1, A_2$ were arbitrary, the claim that $\mathcal{Q}^2(A) \subseteq \mathbf{PS}$ follows. 

{\bf 2.} Suppose now that $\Gamma_A$ is not a forest. $\Gamma_A$ must have a cycle, say $L = (v_1v_2\cdots v_{2r}v_1)$ ($r \geq 2$) and, by Proposition~\ref{propcyclenotrep}, $L$ is not $4$-repeating. Let $L_1$ be the walk obtained from $L$ by inserting $v_1v_2$ after $v_2$ in $L$; let $L_2$ be the walk obtained from $L_1$ by inserting $v_3v_4$ after $v_4$ in $L_1$; and if $r \geq 3$, let $L_3$ be obtained from $L_2$ by inserting $v_5v_6$ after $v_6$ in $L_2$. Observe that $w(L_i) = w(L)$ for each $i$: this follows since the added path at each stage traverses the same edge twice and thus has weight zero. Observe also that each of $L_1$ and $L_2$, and $L_3$ if defined, fails to be $4$-repeating. 

\begin{enumerate}
\item[(a)] If $w(L) = 0$, we define a new closed walk $L'$ as follows: if $|L| \equiv 0 \bmod 8$, let $L' = L$; if $|L| \equiv 2 \bmod 8$, then $L' = L_3$ (this is defined since $|L| \geq 10$); if $|L| \equiv 4 \bmod 8$, $L' = L_2$; if $|L| \equiv 6 \bmod 8$, $L' = L_1$. Observe that in each case $|L'| \equiv 0 \bmod 8$, so $w_4(L') = 0$. 

\item[(b)] If $w(L) = 1$, define $L'$ as follows: if $|L| \equiv 0 \bmod 8$, $L' = L_2$; if $|L| \equiv 2 \bmod 8$, $L' = L_1$; if  $|L| \equiv 4 \bmod 8$, $L' = L$; if  $|L| \equiv 6 \bmod 8$, $L' = L_3$ (this walk is defined since $|L| \geq 6$). Observe that in each case $|L'| \equiv 4 \bmod 8$, so $w_4(L') = 0$. 
\end{enumerate}
In each case $L'$ fails to be $4$-repeating or $4$-odd and so, by Proposition~\ref{kgoodor2kodd}, $\mathcal{Q}^4(A) \not \subseteq \mathbf{P_0}$. 

Write $L = (X_{i_0} Y_{j_0} \cdots X_{i_{r-1}} Y_{j_{r-1}}X_{i_0})$ ($r \geq 2$) with all $i_k$ distinct and all $j_k$ distinct. Define $i_r = i_0$ and $B, C\in \mathcal{Q}_0(A)$ via:
\[
B_{pq} = \left\{\begin{array}{ll}\mathrm{sign}(A_{pq}) & \mbox{if  } p=i_k,q=j_k\,\, (k= 0, \ldots, r-1)\\0&\mbox{otherwise}\end{array}\right. 
\]
\[
C_{pq} = \left\{\begin{array}{ll}\mathrm{sign}(A_{pq}) & \mbox{if  } p=i_{k+1},q=j_k\,\, (k= 0, \ldots, r-1)\\0&\mbox{otherwise}\end{array}\right. 
\]
Multiplying gives:
\begin{eqnarray*}
(BC^{\mathrm{t}})_{pq} &=& \sum_s B_{ps}C_{qs}\\ &=& \left\{\begin{array}{ll}\mathrm{sign}(A_{i_kj_k}A_{i_{k+1}j_k}) & \mbox{if  } p=i_k,q=i_{k+1}\,\, (k= 0, \ldots, r-1)\\0&\mbox{otherwise.}\end{array}\right. 
\end{eqnarray*}
Since the $i_k$ are all distinct it follows, as in the proof of Theorem~\ref{main} that given nonempty $\alpha \subseteq \{1, \ldots, n\}$, the principal minor $(BC^{\mathrm{t}})[\alpha] = 0$ unless $\alpha = \alpha' \stackrel{\text{\tiny def}}{=} \{i_0, \ldots, i_{r-1}\}$, in which case, 
\[
(BC^{\mathrm{t}})[\alpha'] = (-1)^{r+1}\prod_{k=0}^{r-1}(BC^{\mathrm{t}})_{i_ki_{k+1}} = \pm 1\,.
\]
Thus the characteristic polynomial of $BC^{\mathrm{t}}$ is either $\lambda^{n-r}(\lambda^r - 1)$ or $\lambda^{n-r}(\lambda^r + 1)$. So the nonzero eigenvalues of $BC^{\mathrm{t}}$ are just the $r$th roots of $1$ or of $-1$, and since $r \geq 2$, these clearly cannot all lie on the nonnegative real axis. (In fact, if $r \geq 3$ then some of these must be nonreal.) Thus $\mathcal{Q}^2_0(A) \not \subseteq \mathbf{PS}$, and by Remark~\ref{remP0PSclosed}, $\mathcal{Q}^2(A) \not \subseteq \mathbf{PS}$.
\qquad \qed \end{pot2}

\section{Proof of Theorem~\ref{thm_main_caterpillar}}

Before we prove Theorem~\ref{thm_main_caterpillar}, we will show that $\mathcal{Q}^{2r}(A) \subseteq \mathbf{P_0}$ for all $r \in \mathbb{N}$ implies $\mathcal{Q}^{2r}(A) \subseteq \mathbf{PS}$ for all $r \in \mathbb{N}$. 

\begin{definition}
\label{defCsets}
Define for $n,k \in \mathbb{N}$ the following subsets of $\mathbb{C}$.  
\[
\begin{array}{rcl}
F(n) &=& \{z=\rho e^{i\theta} \in \mathbb{C}\colon \rho > 0, |\theta - \pi| < \pi/n\}\\
F_k(n) &=& \{z\in \mathbb{C} \colon z^k \in F(n)\}\\
\mathbb{C}_n' &=& \mathbb{C} \backslash \bigcup_{k \in \mathbb{N}}F_k(n)\\
\mathbb{C}^\mathbb{Q} &=& \{z = \rho e^{2\pi i\theta}\colon \theta \in \mathbb{Q}\}
\end{array}
\]
\end{definition}

\begin{rem}
$F_k(n)$ ($k \geq 2$) is the $k$th preimage of $F(n)$, and $\mathbb{C}_n'$ is the complex plane with $F(n)$ and all its preimages removed. Note that there are in general nonreal elements in $\mathbb{C}_n'$: e.g., if $z = e^{2\pi i/3}$ then $z^k \not \in F(4)$ for any $k > 0$. More generally, with $\theta = 2\pi/(2m+1)$, the set $\{e^{i\theta}, e^{2i\theta}\,\ldots\}$ misses $F(n)$ if $n>2m+1$. 
\end{rem}

We collect together some results and observations involving the sets in Definition~\ref{defCsets}:
\begin{lemma}
\label{lemunionFs}
Let $A \in \mathbb{R}^{n \times n}$. Then 
\begin{enumerate}
\item If $A^k  \in \mathbf{P_0}$ for some $k \in \mathbb{N}$, then $\mathrm{spec}\,A \cap F_k(n) = \emptyset$.
\label{P0spec}
\item If $A^k \in \mathbf{P_0}$ for all $k \in \mathbb{N}$, then $\mathrm{spec}\,A \subseteq \mathbb{C}_n'$.
\item $\mathbb{R}_{\geq 0} \subseteq \mathbb{C}_n'\subseteq \mathbb{C}^\mathbb{Q}$.
\label{CQclaim}
\end{enumerate}
\end{lemma}
\begin{proof}
The first claim with $k=1$ is proved by Kellogg \cite{kellogg}: in fact Kellogg proves that $\lambda \in \mathbb{C}$ is an eigenvalue of an $n \times n$ $P_0$-matrix if and only if $\lambda \not \in F(n)$. The case of general $k$ follows immediately as the eigenvalues of $A^k$ are just the $k$th powers of those of $A$. The second claim is immediate from the first. To verify the final claim observe that: (i) if $z \in \mathbb{R}_{\geq 0}$, then $z^k \in \mathbb{R}_{\geq 0}$ for all $k$ and hence $\mathbb{R}_{\geq 0} \subseteq \mathbb{C}_n'$; (ii) if $z = \rho e^{2\pi i\theta}$ where $\rho > 0$ and $\theta \not \in \mathbb{Q}$, then as is well known
\[
\{e^{2\pi i\theta}, e^{4\pi i\theta}, e^{6\pi i\theta}, \ldots\}
\]
is dense on the unit circle, and so there exists $k\in \mathbb{N}$ such that $z^k \in F(n)$. Thus $\mathbb{C}_n'\subseteq \mathbb{C}^\mathbb{Q}$.
\qquad \qed \end{proof}

\begin{thm}
\label{allpowersP_0}
Let $A \in \mathbb{R}^{n \times m}$. If $\mathcal{Q}^{2r}(A) \subseteq \mathbf{P_0}$ for each $r \in \mathbb{N}$ then $\mathcal{Q}^{2r}(A) \subseteq \mathbf{PS}$ for each $r \in \mathbb{N}$.
\end{thm}
\begin{proof}
Suppose $\mathcal{Q}^{2r}(A) \subseteq \mathbf{P_0}$ for each $r \in \mathbb{N}$. Since in particular $\mathcal{Q}^{2}(A) \subseteq \mathbf{P_0}$, by Theorem~\ref{thmtwoprodP0}, $\myd{\Gamma}_A$ is 2-odd. Fix $r \in \mathbb{N}$ and for brevity let $\mathcal{A} \stackrel{\text{\tiny def}}{=} \mathcal{Q}^{2r}(A) \subseteq \mathbb{R}^{n \times n}$. Given $B \in \mathcal{A}$ and $k \in \mathbb{N}$, $B^k \in \mathcal{Q}^{2kr}(A) \subseteq \mathbf{P_0}$ and so, by Lemma~\ref{lemunionFs}, $\mathrm{Spec}\,\mathcal{A} \subseteq \mathbb{C}_n' \subseteq \mathbb{C}^\mathbb{Q}$. 

Now $\mathcal{A} \cap \mathbf{PS}$ is nonempty as it includes, for example, $C \stackrel{\text{\tiny def}}{=} (AA^{\mathrm{t}})^r$. If $\mathcal{A} \not \subseteq \mathbf{PS}$, then there exists $B \in \mathcal{A}$ with nonreal eigenvalue $\lambda$ (since $\mathrm{Spec}\,\mathcal{A} \subseteq \mathbb{C}_n'$). Choose continuous $\gamma\colon [0,1] \to \mathcal{A}$ such that $\gamma(0) = B$ and $\gamma(1) = C$, possible as $\mathcal{A}$ is path-connected (Lemma~\ref{lempathcon}). As all matrices in $\mathrm{im}(\gamma)$ have exactly $n -\mathrm{rank}(A)$ zero eigenvalues (Lemma~\ref{lemoddrank}), by the continuous dependence of the (nonzero) spectrum of $\gamma(t)$ on $t$, there must be some $t' \in [0,1]$ such that $\gamma(t')$ has a nonzero eigenvalue $z=\rho e^{i\phi}$ with $\rho > 0$ and $\phi$ an irrational multiple of $2\pi$, contradicting the fact that $\mathrm{spec}(\mathcal{A}) \subseteq \mathbb{C}^\mathbb{Q}$. Thus $\mathcal{A} \subseteq \mathbf{PS}$, and since $r$ was arbitrary, $\mathcal{Q}^{2r}(A) \subseteq \mathbf{PS}$ for each $r \in \mathbb{N}$. \qquad \qed \end{proof}

\begin{prop}
\label{inductivefailP0}
If $\mathcal{Q}^{2r}(A) \subseteq \mathbf{P_0}$ for some $r \in \mathbb{N}$, then $\mathcal{Q}^{2s}(A) \subseteq \mathbf{P_0}$ for all $1 \leq s \leq r$. 
\end{prop}
\begin{proof}
We prove the contrapositive. Suppose $\mathcal{Q}^{2s}(A) \not \subseteq \mathbf{P_0}$. By Proposition~\ref{kgoodor2kodd}, there exist $j \geq 2$ and a closed walk $W$ of length $2js$ in $\myd{\Gamma}_A$ which fails to be either $2s$-repeating or $2s$-odd. By Proposition~\ref{propinherit} and Remark~\ref{reminherit} there exists a closed walk of length $2j(s+1)$ on $\myd{\Gamma}_A$ which fails to be either $2(s+1)$-repeating or $2(s+1)$-odd. Applying Proposition~\ref{kgoodor2kodd} again, $\mathcal{Q}^{2(s+1)}(A) \not \subseteq \mathbf{P_0}$. 
\qquad \qed \end{proof}

We are now able to prove Theorem~\ref{thm_main_caterpillar}, restated here for readability.

\begin{center}
\fbox{%
\parbox{8cm}{%
{\bf Theorem~\ref{thm_main_caterpillar}.} The following are equivalent:
\begin{enumerate}
\item[(a)] $\Gamma_A$ is a caterpillar forest.
\item[(b)] $\mathcal{Q}^6(A) \subseteq \mathbf{P_0}$.
\item[(c)] $\mathcal{Q}^4(A) \subseteq \mathbf{PS}$.
\item[(d)] $\mathcal{Q}^{2k}(A) \subseteq \mathbf{P_0}$ for all $k \in \mathbb{N}$.
\item[(e)] $\mathcal{Q}^{2k}(A) \subseteq \mathbf{PS}$ for all $k \in \mathbb{N}$.
\end{enumerate}
}}
\end{center}

\begin{pot3}

We follow the scheme:
\begin{tikzpicture}[baseline=2.8ex, domain=0:4,scale=0.25]
\node at (0.5,2) {$(b)$};
\node[rotate=30] at (2.2,3) {$\Rightarrow$};
\node at (1.7,3.7) {$\scriptstyle{3}$};
\node at (4,4) {$(a)$};
\node at (6,4) {$\Leftarrow$};
\node at (6,4.9) {$\scriptstyle{6}$};
\node at (4,2) {$\Downarrow$};
\node at (4.7,2) {$\scriptstyle{1}$};
\node at (8,4) {$(c)$};
\node at (4,0) {$(d)$};
\node at (6,0) {$\Rightarrow$};
\node at (6,0.9) {$\scriptstyle{4}$};
\node at (8,0) {$(e)$};
\node at (8,2) {$\Uparrow$};
\node at (8.7,2) {$\scriptstyle{5}$};
\node[rotate=60] at (2.2,1) {$\Uparrow$};
\node at (1.8,0.2) {$\scriptstyle{2}$};
\end{tikzpicture}

1) $\Gamma_A$ is a caterpillar forest $\Rightarrow$ $\mathcal{Q}^{2k}(A) \subseteq \mathbf{P_0}$ for all $k \in \mathbb{N}$. Suppose $\Gamma_A$ is a caterpillar forest. Then every closed walk on $\myd{\Gamma}_A$ has weight $0$ (Remark~\ref{remclosedweight}). Moreover $\Gamma_A$ is $2k$-repeating for each $k \in \mathbb{N}$ by Proposition~\ref{onlycatssgood}. By Proposition~\ref{P_0prods}, $\mathcal{Q}^{2k}(A) \subseteq \mathbf{P_0}$ for all $k \in \mathbb{N}$.

2) If $\mathcal{Q}^{2k}(A) \subseteq \mathbf{P_0}$ for all $k \in \mathbb{N}$, then trivially $\mathcal{Q}^6(A) \subseteq \mathbf{P_0}$.

3) $\mathcal{Q}^6(A) \subseteq \mathbf{P_0}$ $\Rightarrow$ $\Gamma_A$ is a caterpillar forest. Suppose $\Gamma_A$ fails to be a caterpillar forest; then either (i) it fails to be a forest, or (ii) it includes a subgraph isomorphic to $T^*$. In case (i) by Theorem~\ref{thm_main_forest}, $\mathcal{Q}^4(A) \not \subseteq \mathbf{P_0}$ and so, by Proposition~\ref{inductivefailP0}, $\mathcal{Q}^6(A) \not \subseteq \mathbf{P_0}$. In case (ii) $A$ has a submatrix, say $B$, such that $\Gamma_B \cong T^*$. Define the walk $W$ as in the proof of Lemma~\ref{notall6rep}: as $W$ fails to be $6$-repeating or $6$-odd, by Proposition~\ref{kgoodor2kodd} $\mathcal{Q}^6(B) \not \subseteq \mathbf{P_0}$. By Proposition~\ref{propnotP0PS}, $\mathcal{Q}^6(A) \not \subseteq \mathbf{P_0}$. 

4) If $\mathcal{Q}^{2k}(A)\subseteq \mathbf{P_0}$ for all $k \in \mathbb{N}$ then, by Theorem~\ref{allpowersP_0}, $\mathcal{Q}^{2k}(A)\subseteq \mathbf{PS}$ for all $k \in \mathbb{N}$.

5) If $\mathcal{Q}^{2k}(A)\subseteq \mathbf{PS}$ for all $k \in \mathbb{N}$, then trivially $\mathcal{Q}^{4}(A)\subseteq \mathbf{PS}$. 

6) Finally, if $\mathcal{Q}^{4}(A)\subseteq \mathbf{PS}$, then $\Gamma_A$ is a caterpillar forest. Suppose $\Gamma_A$ fails to be a caterpillar forest. Either (i) $\Gamma_A$ is not a forest, in which case, by Theorem~\ref{thm_main_forest}, $\mathcal{Q}^4(A) \not \subseteq \mathbf{P_0}$, and hence, by Proposition~\ref{propnotP0notPS} $\mathcal{Q}^{4}(A) \not \subseteq \mathbf{PS}$; or (ii) $A$ has a submatrix, say $B$, such that $\Gamma_B \cong T^*$. It is then easy to find $B_1,B_2,B_3,B_4 \in \mathcal{Q}_0(B)$ such that $B_1B_2^{\mathrm{t}}B_3B_4^{\mathrm{t}} \not \in \mathbf{PS}$ (see Example~\ref{exprod} below), and thus $\mathcal{Q}^4_0(B) \not \subseteq \mathbf{PS}$. By Remark~\ref{remP0PSclosed}, $\mathcal{Q}^4(B) \not \subseteq \mathbf{PS}$ and by Proposition~\ref{propnotP0PS}, $\mathcal{Q}^4(A) \not \subseteq \mathbf{PS}$. 
\qquad \qed \end{pot3}

\begin{example}
\label{exprod}
Define 
\[
B=\left(\begin{array}{ccc}1&0&0\\1&1&1\\0&1&0\\0&0&1\end{array}\right)
\]
so that $\Gamma_B \cong T^*$, and define 
\[
J = \left(\begin{array}{ccc}1&0&0\\1&0&1\\0&1&0\\0&0&0\end{array}\right)
\left(\begin{array}{cccc}1&1&0&0\\0&0&1&0\\0&0&0&1\end{array}\right)
\left(\begin{array}{ccc}0&0&0\\1&1&0\\0&1&0\\0&0&1\end{array}\right)
\left(\begin{array}{cccc}1&0&0&0\\0&1&1&0\\0&1&0&0\end{array}\right) \in \mathcal{Q}^4_0(B)\,.
\]
$J$ has characteristic polynomial $\lambda(\lambda^3-4\lambda^2+3\lambda-1)$ which can easily be computed to have a pair of nonreal roots, using, for example, the implementation of Sturm's theorem in MAXIMA \cite{maxima}. 
\end{example}

\section{Conclusions}

A number of relationships have been presented between the graphs associated with a real matrix $A$ and the products $\mathcal{Q}^k(A)$. Some of the results seem rather surprising, for example, the claim that the apparently weaker condition $\mathcal{Q}^4(A) \subseteq \mathbf{PS}$ implies in fact that $\mathcal{Q}^{2k}(A) \subseteq \mathbf{PS}$ for all $k \in \mathbb{N}$. 

As discussed in the introductory section, the results here have connections with the study of chemical reaction networks. For example, Theorem~\ref{thmtwoprodP0}, which can be derived from results in \cite{banajicraciun}, is related to the question of which chemical reaction systems are incapable of multiple steady states: via results in \cite{gale}, for example, and with some mild additional assumptions, it implies that chemical systems with $2$-odd SR graphs are incapable of multiple equilibria. In this spirit, again with some additional assumptions, a consequence of Theorem~\ref{thm_main_forest} is that chemical systems with acyclic SR graphs have a unique equilibrium which is locally stable. We are unaware of any immediate applications of Theorem~\ref{thm_main_caterpillar}, but mathematically it is the natural next claim after that of Theorem~\ref{thm_main_forest}. Sharper graph-theoretic results are also available, involving more complicated computations on the SR graph and related bipartite graphs. 

\section*{Acknowledgements}
MB's work was supported by EPSRC grant EP/J008826/1 ``Stability and order preservation in chemical reaction networks''. 

\bibliographystyle{unsrt}

\end{document}